\documentclass[12pt,a4paper]{amsart}
\usepackage{amsmath, amssymb}
\usepackage{hyperref}
\usepackage[english]{babel}
\theoremstyle{plain}

\usepackage{mathdots}
\usepackage{tikz-cd}

\usepackage{enumitem}
\usepackage{tikz-cd}

\usepackage{cite}

\newtheorem{theorem}{Theorem}
\newtheorem{lemma}{Lemma}[section]

\newtheorem{example}{Example}

\begin{document}
	
	\title{Lie and Jordan Isomorphisms of \\ Algebras  of Triangular Matrices \\ over Associative Rings}

		\author
		{Oksana Bezushchak}
		\address{Oksana Bezushchak: Faculty of Mechanics and Mathematics, Taras Shevchenko National University of Kyiv, Volodymyrska, 60, Kyiv 01033, Ukraine}
		\email{bezushchak@knu.ua}

	\maketitle
	
	\begin{abstract}
		We describe Lie and Jordan isomorphisms of algebras of triangular matrices over associative rings. 
	\end{abstract}

	\section{Introduction}

Throughout the paper, we consider associative algebras over a commutative associative  ring $\Phi$ containing $1$.

Given two associative algebras $A$ and $B$, a $\Phi$-linear mapping ${\varphi \colon A \to B}$ is called a \textbf{Jordan homomorphism} if
\[	\varphi(x^2)=\varphi(x)^2 \quad \text{for all } x\in A;	\tag{I}
	\] and
\[	\varphi(xyx)=\varphi(x)\varphi(y)\varphi(x) \quad \text{for all } x,y\in A.\tag{II}
\] 

If the algebra $A$ has no additive $2$-torsion, then (II) follows from (I).
For elements $x,y\in A$, denote $x\circ y = xy+yx.$ Assumption (I) implies that $\varphi(x\circ y)=\varphi(x)\circ\varphi(y).$

An $\Phi$-linear mapping $\psi : A\to B$ is called a \textbf{Lie homomorphism} if
\[
\varphi([x,y])=[\varphi(x),\varphi(y)]
\]
for all $x,y\in A$, where $[x,y]=xy-yx$ is the
commutator.

Clearly, a homomorphism of associative algebras
is a Jordan (resp. Lie) homomorphism.

An $\Phi$-linear mapping $\chi : A\to B$ is called an
\textbf{antihomomorphism} if $\chi(xy)=\chi(y)\chi(x)$ for all
elements $x,y\in A$. If $\chi$ is an antihomomorphism,
then $\chi$ is a Jordan homomorphism,
whereas $-\chi$ is a Lie homomorphism.

Following earlier works of L.K.~Hua~\cite{Hua},
N.~Jacobson and C.E.~Rickart \cite{Jacobson2}, I.~Herstein \cite{Her}; M.~Bre\v{s}ar \cite{Bresar,Bresar2} showed that a Jordan
homomorphism onto a semiprime algebra
can be represented as a sum of a homomorphism and an antihomomorphism. For the current state of the theory, see \cite{Bresar_Zelmanov}.

Speaking of Lie isomorphisms, M.~Bre\v{s}ar \cite{Bresar7}
and, later, K.~Beidar, M.~Bre\v{s}ar, M.~Chebotar  and W.S.~Martindale \cite{Bei_Bre_Cheb_Mart_1,Bei_Bre_Cheb_Mart_2,Bei_Bre_Cheb_Mart_3} completely proved Herstein's conjectures
(see \cite{Herstein}) on Lie isomorphisms of prime rings
without additive $2$-torsion.

The paper \cite{Bezushchak_Kashuba_Zelmanov} describes Lie isomorphisms of
rings that are not necessarily semiprime
and $2$-torsion free, but contain three pairwise
orthogonal full idempotents.

Another class of rings that has attracted attention in the literature is that of rings of upper triangular matrices over associative rings. These rings are not semiprime and may not contain full proper idempotents.

Let $R$ be an associative $\Phi$-algebra with $1$, not necessarily commutative, and let $n\geq 3$.
Consider the algebra of upper triangular
$n\times n$ matrices:
	$$ 
T(R) = 
\left(
\begin{array}{cccc}
	\Phi & R & \cdots & R \\
	0 & \ddots & \ddots & \vdots \\
	\vdots & \ddots & \ddots & R \\
	0 & \cdots & 0 & \Phi
\end{array}
\right).$$
Consider also the subalgebras:
\[
\operatorname{Diag}(\Phi)=
\left(
\begin{array}{cccc}
	\Phi & 0 & \cdots & 0 \\
	0 & \ddots & \ddots & \vdots \\
	\vdots & \ddots & \ddots & 0 \\
	0 & \cdots & 0 & \Phi
\end{array}
\right)),
\qquad
 \mathfrak{A}(R)=
\left(
\begin{array}{cccc}
	0 & R & \cdots & R \\
	0 & \ddots & \ddots & \vdots \\
	\vdots & \ddots & \ddots & R \\
	0 & \cdots & 0 & 0
\end{array}
\right),
\]
\[
T(R)=\operatorname{Diag}(\Phi)\ltimes  \mathfrak{A}(R).
\]

Assuming that $\Phi=R$, D.~Dokovi\'c~\cite{Dokovic}
described $\Phi$-linear Lie automorphisms of
the algebra $T(\Phi)$ under the assumption
that $\Phi$ contains no proper idempotents.
In~\cite{Cao1}, Y.~Cao removed this assumption
and extended the description to arbitrary commutative rings $\Phi$ with~$1$.
In~\cite{Cao2}, Y.~Cao described Lie automorphisms of
nilpotent algebras $N(\Phi)$.

K.I.~Beidar, M.~Bre\v{s}ar, and M.A.~Chebotar \cite{BBC} described Jordan isomorphisms of triangular matrix algebras over commutative rings without proper idempotents. C.-K.~Liu and W.-Y.~Tsai~\cite{Liu_Tsai,Liu_Tsai_1} described
Jordan isomorphisms of the algebra $T(\Phi)$,
where $\Phi$ is a commutative ring without $2$-torsion.
D.~Benkovi\v{c}~\cite{Benkovic} described Jordan homomorphisms
of the algebra $T(\Phi)$. In~\cite{Bezushchak2}, we extended
this result to algebras of triangular matrices
$T(R)$, where $R$ is an arbitrary associative
$\Phi$-algebra.

The purpose of this paper is to describe
Lie and Jordan isomorphisms from $T(R)$ to $T(S)$,
where $R$ and $S$ are associative $\Phi$-algebras.

\section{Examples of Isomorphisms}

 For $1\le i,j\le n$ and $a\in R$, let $e_{ij}(a)$ denote the matrix whose $(i,j)$-entry is $a$ and whose remaining entries are zero.

\begin{example}
	Let $R$ and $S$ be associative $\Phi$-algebras,
	and let $\varphi \colon R \to S$ be an isomorphism.
	Then $\varphi$ naturally induces an isomorphism
	\[
	\widetilde{\varphi} \colon T(R) \to T(S), \qquad
	(a_{ij}) \mapsto (\varphi(a_{ij})), \quad a_{ij}\in R.
	\]	\end{example}
	
\begin{example} Let $\psi \colon R \to S$ be an antiisomorphism. Then
	\[
	\widetilde{\psi} \colon T(R) \to T(S), \qquad
	e_{ij}(a) \mapsto e_{n+1-j,n+1-i}(\psi(a)),
	\quad 1\le i,j\le n,\ a\in R,
	\]
	is an antiisomorphism of the algebras $T(R)$
	and $T(S)$.
	
	Let
	\[
	S^{\mathrm{op}}=\{a^{\mathrm{op}} \mid a\in S\}
	\]
	be the same $\Phi$-module as $S$, endowed with the multiplication
	\[
	a^{\mathrm{op}}\cdot b^{\mathrm{op}}=(ba)^{\mathrm{op}}.
	\]
	Then the mapping
	\[
	S \to S^{\mathrm{op}}, \qquad a\mapsto a^{\mathrm{op}},
	\]
	is an antiisomorphism.
	
	The mapping
	\[
	\tau_S \colon T(S)\to T(S^{\mathrm{op}}), \qquad
	e_{ij}(a)\mapsto e_{n+1-j,n+1-i}(a^{\mathrm{op}}),
	\quad 1\le i,j\le n,\ a\in S,
	\]
	is an antiisomorphism of the algebras $T(S)$ and $T(S^{\mathrm{op}})$.
		
	If $\psi$ is an antiisomorphism, then the mapping $-\psi$ is a Lie isomorphism. \end{example}
	
\begin{example} Consider the upper unitriangular group
		\[
		UT(n,S)=
		\left(
		\begin{array}{cccc}
			1_S & S & \cdots & S \\
			0 & \ddots & \ddots & \vdots \\
			\vdots & \ddots & \ddots & S \\
			0 & \cdots & 0 & 1_S
		\end{array}
		\right),
		\]
		where $1_S$ denotes the identity element of $S$. For an
		arbitrary element $g\in UT(n,S)$ the conjugation
		\[
		\widehat{g} : x\to gxg^{-1}, \quad x\in T(S),
		\]
		is an automorphism of the algebra $T(S)$. \end{example}
		
\begin{example} Let $S^{*}$ be the multiplicative group of
		invertible elements of the algebra $S$. Let
		\[
		\operatorname{Diag}(S^{*})=
		\{\operatorname{diag}(u_1,\ldots,u_n)\mid u_i\in S^{*}\}.
		\]
		For any matrix $g\in \operatorname{Diag}(S^{*})$, the
		conjugation
		\[
		\widehat{g} : x\to gxg^{-1}, \quad x\in T(S),
		\]
		is an automorphism of $T(S)$.
\end{example}

	Let $T(R)^{(-)}$ and $T(S)^{(-)}$ denote the adjoint Lie algebras of the associative algebras $T(R)$ and $T(S)$, respectively; that is, the associative multiplication is replaced by the commutator
	\[
	[x,y]=xy-yx.
	\]
	
	Let $I_R$ and $I_S$ be the identity $n\times n$ matrices
	in $T(R)$ and $T(S)$, respectively. Then $\Phi\cdot I_R$ and $\Phi\cdot I_S$
	are central ideals of $T(R)^{(-)}$ and $T(S)^{(-)}$, respectively.

\begin{theorem}\label{Th_1}
Let $R$ and $S$ be associative unital
$\Phi$-algebras, and let $n\geq 3$. Suppose that
\[
\varphi \colon T(R)^{(-)} \slash (\Phi\cdot I_R)\to T(S)^{(-)} \slash (\Phi\cdot I_S)
\]
is an isomorphism of Lie algebras. Then
there exist an idempotent $h\in \Phi$, an
isomorphism $$\psi_1 \colon hR\to hS,$$ an antiisomorphism
$$\psi_2 \colon (1-h)R\to (1-h)S,$$ and invertible
elements $g_1\in UT(n,S)$ and $g_2\in \operatorname{Diag}(S^{*})$ such
that $\varphi$ lifts to
\[
\widehat{g}_1\,\widehat{g}_2\,(\widetilde{\psi}_1-\widetilde{\psi}_2).
\] \end{theorem}

\begin{theorem}\label{Th_2}
Let $R$ and $S$ be associative unital
$\Phi$-algebras, and let $n\geq 3$. Suppose that
\[
\varphi \colon T(R)\to T(S)
\]
is a Jordan isomorphism. Then there exist
an idempotent $h\in\Phi$, an isomorphism
\[
\psi_1 \colon hR\to hS,
\]
an antiisomorphism
\[
\psi_2 \colon (1-h)R\to(1-h)S,
\]
and invertible elements $g_1\in UT(n,S), $ $ g_2\in \operatorname{Diag}(S^{*}),$ such that
\[
\varphi=\widehat{g}_1\,\widehat{g}_2\,(\widetilde{\psi}_1+\widetilde{\psi}_2).
\] \end{theorem}

For $R=\Phi$, the above theorems reduce to
descriptions obtained in \cite{Cao1,Liu_Tsai,Liu_Tsai_1},  respectively.

As far as we know, Theorems 1 and 2 are new
even in the case of automorphisms of $T(R)$, where $R$
is a commutative $\Phi$-algebra, since we do not
assume that the Lie or Jordan automorphism $\varphi$
is $R$-linear. 

In what follows, we always assume that
$R$ and $S$ are associative unital $\Phi$-algebras with identity
elements $1_R$ and $1_S$, respectively, and that $\Phi$ is a commutative
associative ring with identity $1$.

\section{Lie Isomorphisms}\label{Sec_Lie}

Let $\varphi : T(R)/(\Phi\cdot I_R) \to T(S)/(\Phi\cdot I_S)$ be an isomorphism
of Lie algebras. For an element $a\in T(R)$, denote
$\varphi(a)=\varphi(a+\Phi\cdot I_R).$

\begin{lemma}\label{Lm1}
There exists an idempotent $h\in\Phi$
(possibly, $0$ or $1$) such that
\[
\varphi(e_{11}(1_R))=
\left(
\begin{array}{ccccc}
	h & * & \cdots & \cdots & * \\
	0 & 0 & \ddots & \ddots & \vdots \\
	\vdots	 & \ddots & \ddots & \ddots & \vdots \\
	\vdots & \ddots & \ddots & 0 & * \\
	0 & \cdots & \cdots & 0 & h-1_S
\end{array}
\right)+\ \Phi\cdot I_S.
\] \end{lemma}

\begin{proof} Let
\[
\varphi(e_{11}(1_R))=
\left(
\begin{array}{cccc}
	\alpha_1 & * & \cdots & * \\
	0 & \ddots & \ddots & \vdots \\
	\vdots & \ddots & \ddots & * \\
	0 & \cdots & 0 & \alpha_n
\end{array}
\right)+\ \Phi\cdot I_S; \qquad \alpha_1,\ldots,\alpha_n\in\Phi.
\]
Then the following assertions hold:

\medskip

1) For each $1\le i<j\le n$, the element $\alpha_i-\alpha_j$ is
	an idempotent. Indeed, for an arbitrary element
	$u\in T(R)$, we consider the adjoint operator
	\[
\operatorname{ad}(u)\colon T(R)\to T(R), \qquad x\mapsto [u,x].
	\]
	We have
	\[
	\operatorname{ad}(e_{11}(1_R))\,e_{ij}(a)=
	\begin{cases}
		e_{ij}(a), & i=1,\\
		0, & i\ne 1,
	\end{cases}
	\]
	where $1\le i<j\le n$ and $a\in R$. Hence,
	\[
	\operatorname{ad}(e_{11}(1_R))^2=\operatorname{ad}(e_{11}(1_R)).
	\]
	For $1\le i<j\le n$, denote
	\[
	T(S)^+_{ij}=\sum_{\substack{1\le p<q\le n,\\ q-p>j-i}} e_{pq}(S).
	\]
	We have
	\[
	\Bigl(\operatorname{ad}(\varphi(e_{11}(1_R)))^2-\operatorname{ad}(\varphi(e_{11}(1_R)))\Bigr)
	\bigl(e_{ij}(b)+\Phi\cdot I_S\bigr)\in
	\]
\[
\bigl((\alpha_i-\alpha_j)^2-(\alpha_i-\alpha_j)\bigr)e_{ij}(b)+T(S)^+_{ij}+\Phi\cdot I_S \quad \text{for} \quad b\in S.
\] It implies the claim.

\medskip

2) We claim that $\alpha_1-\alpha_n=1_S$.
Indeed, the center of the Lie algebra
\[
\big[T(R)/(\Phi\cdot I_R),\ T(R)/(\Phi\cdot I_R)\big]=\big(\mathfrak{A}(R)+\Phi\cdot I_R\big)/(\Phi\cdot I_R)
\]
is
\[
\big(e_{1n}(R)+\Phi\cdot I_R\big)/(\Phi\cdot I_R).
\]
Hence,
\[
\varphi\bigl((e_{1n}(R)+\Phi\cdot I_R)/(\Phi\cdot I_R)\bigr)
=
(e_{1n}(S)+\Phi\cdot I_S)/(\Phi\cdot I_S).
\]
For an arbitrary element
\[
x\in (e_{1n}(R)+\Phi\cdot I_R)/(\Phi\cdot I_R),
\]
we have $[e_{11}(1_R),x]=x.$ Therefore, for an arbitrary element
\[
y\in \bigl(e_{1n}(S)+\Phi\cdot I_S\bigr)/(\Phi\cdot I_S),
\]
we have $[\varphi(e_{11}(1_R)),y]=y.$ This implies assertion~2).

\medskip

3) We claim that for each $1 \le i < j < k \le n$, $(\alpha_i-\alpha_j)(\alpha_j-\alpha_k)=0.$ 
Indeed,
\[
[e_{11}(1_R),T(R)] =
\left(
\begin{array}{cccc}
	0 & R & \cdots & R \\
	0 & 0 & \cdots & 0 \\
	\vdots & \ddots & \ddots & 0 \\
	0 & \cdots & 0 & 0
\end{array}
\right).
\]
Hence,
\[
\Big[\big[\varphi(e_{11}(1_R)),T(S)\big],\big[\varphi(e_{11}(1_R)),T(S)\big]\Big]
\subseteq \Phi\cdot I_S.
\]
In particular,
\[
\Big[\big[\varphi(e_{11}(1_R)),e_{ij}(1_S)\big],\big[\varphi(e_{11}(1_R)),e_{jk}(1_S)\big]\Big]
\in \Phi\cdot I_S.
\]
The left-hand side lies in
\[
(\alpha_i-\alpha_j)(\alpha_j-\alpha_k)e_{ik}(1_S)+T_{ik}^{+}(S).
\]
This implies claim~3).

\medskip

4) Let $\alpha_1=\cdots=\alpha_k$ and $\alpha_k\ne \alpha_{k+1}$. Then either $k=1$
or $k=n-1$.

Clearly, $k\le n-1$, since the operator $\operatorname{ad}(\varphi(e_{11}(1_R)))$ is
not nilpotent. Suppose that $2\le k\le n-2$.
Since
\[
\varphi(e_{11}(1_R))+\Phi\cdot I_S=\varphi(e_{11}(1_R))-\alpha_{k+1}I_S+\Phi\cdot I_S,
\]
we may assume, without loss of generality, that
\[
\alpha_1=\cdots=\alpha_k\ne 0,\qquad \alpha_{k+1}=0.
\]
By claim 1), $h=\alpha_1$
is an idempotent.

For any $k+2\le j\le n$, claim 3) implies $(h-0)(0-\alpha_j)=0,$ hence $\alpha_{k+2},\ldots,\alpha_n\in(1-h)\Phi$.
Now,
\[h\, \varphi(e_{11}(1_R))=
\begin{tikzcd}[column sep=-0.5em, row sep=0pt, ampersand replacement=\&]
	\left(
	\begin{array}{cccccc}
		h      & \cdots & *     & *      & \cdots & * \\
	0      & \ddots & \vdots & \vdots &        & \vdots \\
		\vdots      & \cdots & h      & *      & \cdots & * \\
	0      & \cdots & 0      & 0      & \cdots & * \\
		\vdots &        & \vdots & \vdots & \ddots & \vdots \\
	0     & \cdots & 0     & 0      & \cdots & 0
	\end{array}
	\right)
	\hspace{-0.6em}
	\begin{array}{c}
		\left.\vphantom{
			\begin{array}{c}
				h\\ h\\ h
			\end{array}
		}\right\} {\scriptstyle k} 
		\\[+12.25ex]
	\end{array}
	\& {}
\end{tikzcd}
+\ \Phi\cdot I_S.
\]
Let
\[
h\,\varphi(e_{22}(1_R))=
\left(
\begin{array}{cccc}
	\beta_1 & * & \cdots & * \\
	0 & \ddots & \ddots & \vdots \\
	\vdots & \ddots & \ddots & * \\
	0 & \cdots & 0 & \beta_n
\end{array}
\right)+\ \Phi\cdot I_S; \quad \beta_i\in h\Phi,\quad 1\le i\le n.
\]
We have:
\[
\big[[T(R),e_{11}(1_R)],e_{22}(1_R)\big]=e_{12}(R), \quad 
e_{12}(R)\cap [\mathfrak{A}(R),\mathfrak{A}(R)]=(0).
\]
This implies that
\[
\big[[T(S),h\, \varphi(e_{11}(1_R))],\, h\, \varphi(e_{22}(1_R))\big]
\cap\ [\mathfrak{A}(S),\mathfrak{A}(S)]=(0).
\]

For all  $1\le i\le k$ and $i+2\le j\le n$, we have
\[
\left[ \ 
\left[
e_{ij}(1_S),
\begin{tikzcd}[column sep=-0.5em, row sep=0pt, ampersand replacement=\&]
	\left(
	\begin{array}{cccccc}
		h      & \cdots & *     & *      & \cdots & * \\
		0      & \ddots & \vdots & \vdots &        & \vdots \\
		\vdots      & \cdots & h      & *      & \cdots & * \\
		0      & \cdots & 0      & 0      & \cdots & * \\
		\vdots &        & \vdots & \vdots & \ddots & \vdots \\
		0     & \cdots & 0     & 0      & \cdots & 0
	\end{array}
	\right)
\end{tikzcd} 
\right] , \left(
\begin{array}{cccc}
	\beta_1 & * & \cdots & * \\
	0 & \ddots & \ddots & \vdots \\
	\vdots & \ddots & \ddots & * \\
	0 & \cdots & 0 & \beta_n
\end{array}
\right)
\right] \in
\]
\[
 (\beta_i-\beta_j)e_{ij}(h)+T(S)^+_{ij}.
\]  Hence, $\beta_i-\beta_j=0$.

Since $k\le n-2$, it follows that $\beta_1=\cdots=\beta_k=\beta_{k+2}=\cdots=\beta_n.$ Also, since $k\ge 2$, it follows that $\beta_1=\beta_{k+1}.$

We have proved that $\beta_1=\cdots=\beta_n$, which implies
that the operator $\operatorname{ad}\big(\varphi(e_{22}(h))\big)$ is nilpotent.
Hence, the operator $\operatorname{ad}(e_{22}(h))$ is also nilpotent,
a contradiction.

\medskip

5) Suppose now that $\alpha_n=\alpha_{n-1}=\cdots=\alpha_{n-q+1},$ $\alpha_{n-q}\ne \alpha_{n-q+1}.$ We claim that either $q=1$ or $q=n-1$.
To prove this, consider the composition
\[
T(R)/(\Phi\cdot I_R) \xrightarrow{\ \varphi\ } T(S)/(\Phi\cdot I_S)
\xrightarrow{\ \tau_S\ } T(S^{\mathrm{op}})/(\Phi\cdot I_{S^{\mathrm{op}}}),\quad 
\pi=\tau_S\varphi.
\]
Then
\[
\pi(e_{11}(1_R))=
\left(
\begin{array}{cccc}
	\alpha_n & * & \cdots & * \\
	0 & \ddots & \ddots & \vdots \\
	\vdots & \ddots & \ddots & * \\
	0 & \cdots & 0 & \alpha_1
\end{array}
\right).
\]
Thus,   claim 5) follows from  claim 4).

\medskip

6) Suppose that $k=n-1$. Then
\[
\left(
\begin{array}{cccc}
	\alpha_1 & * & \cdots & * \\
	0 & \ddots & \ddots & \vdots \\
	\vdots & \ddots & \ddots & * \\
	0 & \cdots & 0 & \alpha_n
\end{array}
\right)+\Phi\cdot I_S
=
\left(
\begin{array}{cccc}
	0 & * & \cdots & * \\
	0 & \ddots & \ddots & \vdots \\
	\vdots & \ddots & 0 & * \\
	0 & \cdots & 0 & \alpha_n-\alpha_1
\end{array}
\right)+\Phi\cdot I_S.
\] By  claim 2), $\alpha_n-\alpha_1=-1_S.$

\medskip

7) Similarly, if $q=n-1$, then
\[
\left(
\begin{array}{cccc}
	\alpha_1 & * & \cdots & * \\
	0 & \ddots & \ddots & \vdots \\
	\vdots & \ddots & \ddots & * \\
	0 & \cdots & 0 & \alpha_n
\end{array}
\right)+\Phi\cdot I_S
=
\left(
\begin{array}{cccc}
	\alpha_1-\alpha_n & * & \cdots & * \\
	0 & 0 & \ddots & \vdots \\
	\vdots & \ddots & \ddots & * \\
	0 & \cdots & 0 & 0
\end{array}
\right)+\Phi\cdot I_S
=\] 
\[\left(
\begin{array}{cccc}
	1_S & * & \cdots & * \\
	0 & 0 & \ddots & \vdots \\
	\vdots & \ddots & \ddots & * \\
	0 & \cdots & 0 & 0
\end{array}
\right)+\Phi\cdot I_S.
\]

 \medskip
 
 8) Suppose that $k=q=1$. In particular, $\alpha_1\ne \alpha_2$.
We have
\[
\left(
\begin{array}{cccc}
	\alpha_1 & * & \cdots & * \\
	0 & \ddots & \ddots & \vdots \\
	\vdots & \ddots & \ddots & * \\
	0 & \cdots & 0 & \alpha_n
\end{array}
\right)+\Phi\cdot I_S
=\]
\[\left(
\begin{array}{cccc}
	\alpha_1-\alpha_2 & * & \cdots & * \\
		0 & 0 & * & * \\
	0 & 0 & 	\alpha_3-\alpha_2  & \vdots \\
	\vdots &  \ \ \ \ \ddots  & \ \ \ \ \ddots & * \\
	0 & \ \ \ \cdots & 0 & \alpha_n-\alpha_2
\end{array}
\right)+\Phi\cdot I_S.
\]
By claim 1), $h=\alpha_1-\alpha_2$ is an idempotent, and by claim 2), $\alpha_n-\alpha_2=h-1_S.$ As shown in claim 4), $\alpha_i-\alpha_2\in (1-h)\Phi$  for $ 3\le i\le n.$ Moreover, $\alpha_{n-1}\ne \alpha_n$.

Consider the ideals
\[
I' = hT(R)+(\alpha_{n-1}-\alpha_n)T(R) \quad \text{and} \quad 
I'' = hT(S)+(\alpha_{n-1}-\alpha_n)T(S).
\]
Clearly,
\[
\varphi\bigl((I'+\Phi\cdot I_R)/(\Phi\cdot I_R)\bigr)
=
(I''+\Phi\cdot I_S)/(\Phi\cdot I_S).
\]
Hence,
\[
\overline{\varphi}:T(R)/(I'+\Phi\cdot I_R )\to T(S)/(I''+\Phi\cdot I_S)
\]
is a Lie isomorphism. We have
\[
\overline{\varphi}(e_{11}(1_R))=
\left(
\begin{array}{cccc}
	0 & * & \cdots & * \\
	\vdots & \ddots & \ddots & \vdots \\
		0 & \cdots & \alpha & * \\
	0 & \cdots & 0 & \alpha
\end{array}
\right)
+\overline{\varphi}(\Phi)\cdot 1_{\,S/I''}.
\]
By claim  4), $\overline{\varphi}(e_{11}(1_R))=0,$ which implies that $$\alpha_{n-1}-\alpha_n=(\alpha_{n-1}-\alpha_2)-(\alpha_n-\alpha_2)=1-h.$$

On the other hand, by claim 2), $\alpha_2-\alpha_n=1-h.$ Hence, $\alpha_{n-1}-\alpha_2=0.$ If $\alpha_i-\alpha_2\ne 0$ for $ 3\le i\le n-2,$ then $ (\alpha_i-\alpha_2-0)(0-(\alpha_n-\alpha_2))\ne 0,$ which contradicts  claim 3).

Now,
\[
\varphi(e_{11}(1_R))=
\left(
\begin{array}{cccc}
	\alpha_1-\alpha_2 & * & \cdots & * \\
	0 & 0 & \ddots & \vdots \\
	\vdots & \ddots & 0 & * \\
	0 & \cdots & 0 & \alpha_n-\alpha_2
\end{array}
\right)+\Phi\cdot I_S
=\]
\[\left(
\begin{array}{cccc}
	h & * & \cdots & * \\
	0 & 0 & \ddots & \vdots \\
	\vdots & \ddots & 0 & * \\
	0 & \cdots & 0 & h-1
\end{array}
\right)+\Phi\cdot I_S.
\]
This completes the proof of the lemma. \end{proof}

Let $h \in \Phi$ be an idempotent. Consider the ideal
\[
hT(R)=
\left(
\begin{array}{cccc}
	h\Phi & hR & \cdots & hR \\
	0 & \ddots & \ddots & \vdots \\
	\vdots & \ddots & h\Phi & hR \\
	0 & \cdots & 0 & h\Phi
\end{array}
\right)
\]
of the algebra $T(R)$. Let
\[
\varphi_h \colon
hT(R)/ (h\Phi\cdot I_R)
\longrightarrow
hT(S)/ (h\Phi \cdot I_S)
\]
be the restriction of the Lie isomorphism $\varphi$.

\begin{lemma}\label{Lm2}
Suppose that
\[
\varphi\bigl(e_{11}(h)\bigr)=
\left(
\begin{array}{cccc}
	h & * & \cdots & * \\
	0 & 0 & \ddots & \vdots \\
	\vdots & \ddots & \ddots & * \\
	0 & \cdots & 0 & 0
\end{array}
\right)
+ h\Phi\cdot I_S.
\]
Then there exist an isomorphism $\psi \colon hR \to hS$ and invertible elements $g_1 \in UT(h,S) $ and $ g_2 \in \operatorname{Diag}(S^*),$ such that $\varphi_h$ lifts to $$\hat{g}_1\,  \hat{g}_2 \, \tilde{\psi}.$$ \end{lemma}

We say that Lie isomorphisms
\[
\varphi_1,\varphi_2:\ hT(R)/(h\Phi\cdot I_R) \to hT(S)/(h\Phi\cdot I_S)
\]
are \textbf{equivalent} if there exists $g\in UT(n,S)\cdot \operatorname{Diag}(S^{*})$ such that ${\varphi_2=\widehat{g}\,\varphi_1.}$ 

\begin{proof} Let
\[
\varphi(e_{11}(h))=
\left(
\begin{array}{cccc}
	h & a_{12} & \cdots & a_{1n} \\
	0 & 0 & \ddots & \vdots \\
	\vdots & \ddots & \ddots & * \\
	0 & \cdots & 0 & 0
\end{array}
\right)+h\Phi\cdot I_S,
\]
where $a_{12},\ldots,a_{1n}\in hS.$ Let
\[
x=
\left(
\begin{array}{cccc}
	1_S & a_{12} & \cdots & a_{1n} \\
	0 & 1_S & 0 & 0 \\
	\vdots & \ddots & \ddots & 0 \\
	0 & \cdots & 0 & 1_S
\end{array}
\right).
\]
Then
\[
x\,\varphi(e_{11}(h))\,x^{-1}=
\left(
\begin{array}{cccc}
	h & 0 & 0 \ \  \cdots & 0 \\
	0 & 0 & * \ \ \cdots & * \\
	\vdots & \ddots & \ddots & * \\
	0 & \cdots & 0 & 0
\end{array}
\right)+h\Phi\cdot I_S.
\]
Hence, up to equivalence, we may assume that
\begin{equation}\label{EQ1}
\varphi(e_{11}(h))=
\left(
\begin{array}{ccccc}
	h & 0 &  0 &\cdots & 0 \\
	0 & 0 & * & \cdots & * \\
	\vdots & \ddots & \ddots & \ddots & \vdots \\
		0 & 0 & \ddots & 0 & * \\
	0 & 0& \cdots & 0 & 0
\end{array}
\right)+h\Phi\cdot I_S.\end{equation}
Suppose that $1 \le t \le n-1$ and that condition $P(t)$ holds, namely,
\[
\varphi(e_{ii}(h))=
\begin{tikzcd}[column sep=-0.5em, row sep=0pt, ampersand replacement=\&]
	\left(
	\begin{array}{c|c}
		e_{ii}(h) & 0 \\
		\hline
		0 &
		\begin{array}{ccc}
			0 & \cdots & * \\
			\vdots & \ddots & \vdots \\
			0 & \cdots & 0
		\end{array}
	\end{array}
	\right)
	\hspace{-0.6em}
	\begin{array}{c}
		\Bigl.\vphantom{
			\begin{array}{c}
				h\\ h\\ h
			\end{array}
		}\bigr\} {\scriptstyle t} 
		\\[+12.0ex]
	\end{array}
	\& {}
\end{tikzcd}
+h\Phi\cdot I_S
\]
for $i=1,2,\ldots,t$.

As shown above, condition $P(1)$ holds. We will show that $\varphi$ is equivalent to a Lie isomorphism satisfying $P(t+1)$.

Let
\[
\varphi(e_{t+1,t+1}(h))=
\left(
\begin{array}{cccc}
	\gamma_1 & * & \cdots & * \\
	0 & \ddots & \ddots & \vdots \\
	\vdots & \ddots & \ddots & * \\
	0 & \cdots & 0 & \gamma_n
\end{array}
\right)+h\Phi\cdot I_S.
\]
We have
\[
\big[\, [T(R),e_{11}(h)],e_{t+1,t+1}(h)\, \big]=e_{1,t+1}(hR),
\]
\[
e_{1,t+1}(hR)\subseteq \mathfrak{A}(R)^t,\qquad e_{1,t+1}(hR)\cap \mathfrak{A}(R)^{t+1}=(0).
\]
This implies that
\[
X=\big[\, [T(S),\varphi(e_{11}(h))],\varphi(e_{t+1,t+1}(h))\, \big]\subseteq \mathfrak{A}(S)^t, \quad X\cap \mathfrak{A}(S)^{t+1}=(0).
\]

For $2\le j\le n$, we have
\[
[\, e_{1j}(h),\varphi(e_{11}(h))\, ]\in -e_{1j}(h)+T^+_{1j}(S),
\]
\[
[\, e_{1j}(h),\varphi(e_{t+1,t+1}(h))\, ]\in (\gamma_j-\gamma_1)e_{1j}(h)+T^+_{1j}(S).
\]
If $j \le t$ and $\gamma_j-\gamma_1 \ne 0$, then $X \not\subseteq \mathfrak{A}(S)^t$.
If $j \ge t+2$, then $X \subseteq \mathfrak{A}(S)^{t+1}$.
Hence, for every $j$ such that $1 \le j \le n$ and $j \ne t+1$, we have $\gamma_j=\gamma_1.$ Therefore,
\[
y=\varphi\big(e_{t+1,t+1}(h)\big)=\]
\[\left(
\begin{array}{ccccccc}
	0 & * & * &  \cdots & & \cdots & * \\
	0 & 0 & * & \cdots & & \cdots & * \\
		\vdots & \ddots & \ddots & \ddots & & \cdots & *\\
	0 & 0  &  \cdots &  \gamma_{t+1}\!-\!\gamma_1 & \ddots & \cdots & * \\
		0 & 0 & \ddots & \ddots & 0 & \ddots & *\\
		\vdots & \vdots &   & \ddots & \ddots & \ddots & *\\
	0 & 0 & \cdots & \cdots & \cdots &  & 0
\end{array}
\right)\!{\scriptstyle (t+1)\text{-th row}} \ 
+\, h\Phi\cdot I_S.
\]
Denote $\gamma_{t+1}-\gamma_1=\gamma.$

Let $C(hR)$ denote the centralizer of $e_{ii}(h)$, where $1 \le i \le t$, in $h\mathfrak{A}(R)$. It is easy to see that
\[
C(hR)=
\begin{tikzcd}[column sep=-0.5em, row sep=0pt, ampersand replacement=\&]
	\left(
	\begin{array}{c|c}
	0& 0 \\
	\hline
		0 & *
	\end{array}
	\right)
	\hspace{-0.6em}
	\begin{array}{c}
		\Bigl.\vphantom{
			\begin{array}{c}
				h\\ h\\ h
			\end{array}
		}\bigr\} {\scriptstyle t} 
		\\[+5.0ex]
	\end{array}
	\& {}
\end{tikzcd}.\]
Let us show that the centralizer of $t$ elements $\varphi(e_{ii}(h))$, where $1 \le i \le t$, is
\[
C(hS)=
\begin{tikzcd}[column sep=-0.5em, row sep=0pt, ampersand replacement=\&]
	\left(
	\begin{array}{c|c}
		0& 0 \\
		\hline
		0 & *
	\end{array}
	\right)
	\hspace{-0.6em}
	\begin{array}{c}
		\Bigl.\vphantom{
			\begin{array}{c}
				h\\ h\\ h
			\end{array}
		}\bigr\} {\scriptstyle t} 
		\\[+5.0ex]
	\end{array}
	\& {}
\end{tikzcd}.
\]

Let
\[
\sum_{i=1}^{t}\varphi(e_{ii}(h))=
\left(
\begin{array}{c|c}
	I_t(h) & 0 \\
	\hline
	0 & A
\end{array}
\right),
\]
where
\[
I_t(h)=
\left(
\begin{array}{ccc}
	h & \cdots & 0 \\
	\vdots & \ddots & \vdots \\
	0 & \cdots & h
\end{array}
\right)
\hspace{-0.6em}
\begin{array}{c}
	\left.\vphantom{
\begin{array}{ccc}
	h & \cdots & 0 \\
	\vdots & \ddots & \vdots \\
	0 & \cdots & h
\end{array}
	}\right\}
	\\[-3.2ex] \ 
\end{array} \!\!{\scriptstyle t} 
\]
and $A$ is an $(n-t)\times(n-t)$ upper triangular matrix with zero diagonal entries over $hS$.

Let
\[
\left(
\begin{array}{c|c}
	a & b \\
	\hline
	0 & d
\end{array}
\right)\in \mathfrak{A}(hS)
\]
be an element of the centralizer of
\[
\left(
\begin{array}{c|c}
	I_t(h) & 0 \\
	\hline
	0 & A
\end{array}
\right),
\]
that is,
\[
\left(
\begin{array}{c|c}
	I_t(h) & 0 \\
	\hline
	0 & A
\end{array}
\right)
\left(
\begin{array}{c|c}
	a & b \\
	\hline
	0 & d
\end{array}
\right)
=
\left(
\begin{array}{c|c}
	a & b \\
	\hline
	0 & Ad
\end{array}
\right),
\]
\[
\left(
\begin{array}{c|c}
	a & b \\
	\hline
	0 & d
\end{array}
\right)
\left(
\begin{array}{c|c}
	I_t(h) & 0 \\
	\hline
	0 & A
\end{array}
\right)
=\left(
\begin{array}{c|c}
	a & b A \\
	\hline
	0 & dA
\end{array}
\right) \]
\[\text{and} \quad
\left(
\begin{array}{c|c}
	a & b \\
	\hline
	0 & Ad
\end{array}
\right)=
\left(
\begin{array}{c|c}
	a & b A \\
	\hline
	0 & dA
\end{array}
\right). \quad \text{Hence,} \quad b=b A.
\]
Since the matrix $A$ is nilpotent, we conclude
that $b=0$. The matrix $a$ commutes with $e_{ii}(h)$ for $1\le i\le t$; hence, $a$ is diagonal. Now,
\[
\left(
\begin{array}{c|c}
	a & b \\
	\hline
	0 & d
\end{array}
\right)\in \mathfrak{A}(S)
\]
implies that $a=0$. 

Thus, we have proved that
\[
\varphi\bigl(\, (C(hR)+h\Phi\cdot I_R)/(h\Phi\cdot I_R) \, \bigr)
=
\big( C(hS)+h\Phi\cdot I_S\big)/(h\Phi\cdot I_S),
\]
and therefore
\[y=\varphi\big(\,e_{t+1,t+1}(h)\,\big)=\begin{tikzcd}[column sep=-0.5em, row sep=0pt, ampersand replacement=\&]
\left(
\begin{array}{c|c}
	0 & 0 \\
	\hline
	0 &
	\begin{array}{cccc}
		\gamma & * & \cdots & * \\
		0 & 0 & \ddots & \vdots \\
		\vdots & \ddots & \ddots & * \\
		0 & \cdots & 0 & 0
	\end{array}
\end{array}
\right)
	\hspace{-0.6em}
	\begin{array}{c}
		\Bigl.\vphantom{
			\begin{array}{c}
				h\\ h\\ h
			\end{array}
		}\bigr\} {\scriptstyle t} 
		\\[+17.3ex]
	\end{array}
	\& {}
\end{tikzcd}
+\, h\Phi\cdot I_S.
\]

As in the proof of Lemma~\ref{Lm1}, we observe that
\[
\operatorname{ad}\big(e_{t+1,t+1}(h)\big)\, \bigl(\operatorname{ad}(e_{t+1,t+1}(h))-1\bigr)\,C(hR)=(0).
\]
This implies that
\[
\operatorname{ad}(y)\bigl(\operatorname{ad}(y)-1\bigr)\,C(hS)=(0), \quad \gamma=h,
\]
and
\[
y=
\left(
\begin{array}{c|c}
	0 & 0 \\
	\hline
	0 & d
\end{array}
\right),
\ \ 
d=
\left(
\begin{array}{ccccc}
	h & a_{12} & & \cdots & a_{1,n-t-1} \\
	0 & 0 & * & \cdots& * \\
	\vdots & \ddots &\ddots &  \ddots & \vdots \\
	\vdots	 &  & \ddots & \ddots & * \\
	0 & \cdots & & 0 & 0
\end{array}
\right)
\hspace{-0.6em}
\begin{array}{c}
	\left.\vphantom{
	\begin{array}{ccccc}
		h & a_{12} & & \cdots & a_{1,n-t-1} \\
		0 & 0 & * & \cdots& * \\
		\vdots & \ddots &\ddots &  \ddots & \vdots \\
		\vdots	 &  & \ddots & \ddots & * \\
		0 & \cdots & & 0 & 0
	\end{array}
	}\right\}
	\\[-3.2ex] \ 
\end{array} \!\!{\scriptstyle n-t}  \, ,
\quad a_{ij}\in hS.
\]

Let
\[
x=
\left(
\begin{array}{ccccc}
	1_S & a_{12} & & \cdots & a_{1,n-t-1} \\
		0 & 1 & 0 & \cdots& 0 \\
	\vdots & \ddots &\ddots &  \ddots & \vdots \\
\vdots	 &  & \ddots & \ddots & 0 \\
	0 & \cdots & & 0 & 1
\end{array}
\right).
\]
Then
\[
x d x^{-1}=
\left(
\begin{array}{ccccc}
	h & 0 & 0 & \cdots & 0 \\
	0 & 0 & * & \cdots& * \\
	\vdots & \ddots &\ddots &  \ddots & \vdots \\
	\vdots	 &  & \ddots & \ddots & * \\
	0 & \cdots & & 0 & 0
\end{array}
\right)
\hspace{-0.6em}
\begin{array}{c}
	\left.\vphantom{
		\begin{array}{ccccc}
			h & 0 & 0 & \cdots & 0 \\
			0 & 0 & * & \cdots& * \\
			\vdots & \ddots &\ddots &  \ddots & \vdots \\
			\vdots	 &  & \ddots & \ddots & * \\
			0 & \cdots & & 0 & 0
		\end{array}
	}\right\}
	\\[-3.2ex] \ 
\end{array} \!\!{\scriptstyle n-t} \, ,
\]
and therefore
\[
\left(
\begin{array}{c|c}
	I_t & 0 \\
	\hline
	0 & x
\end{array}
\right)
y
\left(
\begin{array}{c|c}
	I_t & 0 \\
	\hline
	0 & x
\end{array}
\right)^{-1}
=\begin{tikzcd}[column sep=-0.5em, row sep=0pt, ampersand replacement=\&]
	\left(
	\begin{array}{c|c}
			e_{t+1,t+1}(h) & 0 \\
		\hline
		0 & B
	\end{array}
	\right)
	\hspace{-0.6em}
	\begin{array}{c}
		\Bigl.\vphantom{
			\begin{array}{c}
				h\\ h\\ h
			\end{array}
		}\bigr\} {\scriptstyle t+1} 
		\\[+5.0ex]
	\end{array}.
	\& {}
\end{tikzcd}
\]
We observe that
\[
\left(
\begin{array}{c|c}
	I_t & 0 \\
	\hline
	0 & x
\end{array}
\right)
\left(
\begin{array}{c|c}
	e_{ii}(h) & 0 \\
	\hline
	0 & *
\end{array}
\right)
\left(
\begin{array}{c|c}
	I_t & 0 \\
	\hline
	0 & x
\end{array}
\right)^{-1}
=
\left(
\begin{array}{c|c}
	e_{ii}(h) & 0 \\
	\hline
	0 & *
\end{array}
\right),
\qquad 1\le i\le t,
\]
as before.

Now, up to equivalence, we may assume that
\[
\varphi(e_{ii}(h))=\begin{tikzcd}[column sep=-0.5em, row sep=0pt, ampersand replacement=\&]
	\left(
	\begin{array}{c|c}
		e_{ii}(h) & 0 \\
		\hline
		0 & *
	\end{array}
	\right)
	\hspace{-0.6em}
	\begin{array}{c}
		\Bigl.\vphantom{
			\begin{array}{c}
				h\\ h\\ h
			\end{array}
		}\bigr\} {\scriptstyle t} 
		\\[+5.0ex]
	\end{array},
	\&  {}
\end{tikzcd}
\qquad 1\le i\le t,
\]
and
\[
\varphi(e_{t+1,t+1}(h))=
\begin{tikzcd}[column sep=-0.5em, row sep=0pt, ampersand replacement=\&]
	\left(
	\begin{array}{c|c}
		e_{t+1,t+1}(h) & 0 \\
		\hline
		0 & *
	\end{array}
	\right)
	\hspace{-0.6em}
	\begin{array}{c}
		\Bigl.\vphantom{
			\begin{array}{c}
				h\\ h\\ h
			\end{array}
		}\bigr\} {\scriptstyle t+1} 
		\\[+5.0ex]
	\end{array}.
	\&  {}
\end{tikzcd}
\]

Let $1\le i\le t $ and 
\[
\varphi(e_{ii}(a))\!=\!\!\!
\begin{tikzcd}[column sep=-0.5em, row sep=0pt, ampersand replacement=\&]
	\left(
	\begin{array}{c|c}
		e_{i i}(h) & a \\
		\hline
		0 & d
	\end{array}
	\right)
	\hspace{-0.6em}
	\begin{array}{c}
		\Bigl.\vphantom{
			\begin{array}{c}
				h\\ h\\ h
			\end{array}
		}\bigr\} {\scriptstyle t+1} 
		\\[+5.0ex]
	\end{array} \!+\!\!
	\&  {}
\end{tikzcd}
h\Phi\cdot I_S,
\ \ a\!=\!\! \left(
\begin{array}{cccc}
	0 & \cdots & 0 & 0 \\
	\vdots & \ddots &   & \vdots \\
	0 & \cdots & 0 & 0 \\
	* & \cdots & * & *
\end{array}
\right)
\hspace{-0.8em}
\begin{array}{c}
	\left.\vphantom{
	\begin{array}{cccc}
		0 & \cdots & 0 & 0 \\
		\vdots & \ddots &   & \vdots \\
		0 & \cdots & 0 & 0 \\
		* & \cdots & * & *
	\end{array}
	}\right\}
	\\[-3.2ex] \ 
\end{array} \!\!{\scriptstyle t+1}   \ .
\]
The matrix
\[
\left(
\begin{array}{c|c}
	e_{ii}(h) & a \\
	\hline
	0 & d
\end{array}
\right) \quad \text{commutes with} \quad
\begin{tikzcd}[column sep=-0.5em, row sep=0pt, ampersand replacement=\&]
	\left(
	\begin{array}{c|c}
		e_{t+1,t+1}(h) & 0 \\
		\hline
		0 & B
	\end{array}
	\right)
	\hspace{-0.6em}
	\begin{array}{c}
		\Bigl.\vphantom{
			\begin{array}{c}
				h\\ h\\ h
			\end{array}
		}\bigr\} {\scriptstyle t+1} 
		\\[+5.0ex]
	\end{array}.
	\&  {}
\end{tikzcd}
\]
This implies that $e_{t+1,t+1}(h)\,a=aB,$ and we have $e_{t+1,t+1}(h)\,a=a.$
Since the matrix \(B\) is nilpotent, we conclude that
\[a=0,\qquad
\varphi\big(e_{ii}(h)\big)=
\left(
\begin{array}{c|c}
	e_{ii}(h) & 0 \\
	\hline
	0 & *
\end{array}
\right)+h\Phi\cdot I_S, \quad 
1\le i\le t+1.
\]

We have proved that, up to equivalence, one may always assume condition $P(t+1)$, and hence condition $P(n)$. In other words, we assume that
\[
\varphi\big(e_{ii}(h)\big)=e_{ii}(h)+h\Phi\cdot I_S,\qquad 1\le i\le n.
\]
Since
\[
e_{ij}(R)=\big[[T(R),e_{ii}(1_R)],e_{jj}(1_R)\big], \qquad 1\le i<j\le n,
\]
it follows that
\[
\varphi\big(e_{ij}(hR)\big)=\big(e_{ij}(hS)+h\Phi\cdot I_S\big)/(h\Phi\cdot I_S),
\qquad 1\le i<j\le n.
\]

Let
\[
\varphi\bigl(e_{ij}(h)\bigr)=e_{ij}(u_{ij})+h\Phi\cdot I_S,\qquad 1\le i<j\le n,\quad u_{ij}\in hS.
\]
The equalities
\[
e_{1n}(hR)=\big[[e_{1i}(h),e_{ij}(h)],e_{jn}(R)\big],
\]
\[
e_{1n}(hS)=\big[[e_{1i}(u_{1i}),e_{ij}(u_{ij})],e_{jn}(hS)\big]
\]
imply that, for all $1\le i<j\le n$, the elements $u_{ij}$ are invertible in $hS$.

Now consider the idempotent $f=1-e\in\Phi$. Let
\[
x_1=1_S,\qquad x_2=u_{12}+f,\qquad x_3=u_{12}u_{23}+f,\quad  \ldots,\]
\[ x_n=
u_{12}u_{23}\cdots u_{n-1,n}+f, \quad
x=\operatorname{diag}(1_S,x_2,\ldots,x_n)\in
\operatorname{Diag}(S^{*}).
\]
Then
\[
x\,e_{i,i+1}(u_{i,i+1})\,x^{-1}=e_{i,i+1}(h).
\]
Hence, up to equivalence, one may assume that
\[
\varphi\bigl(e_{i,i+1}(h)\bigr)=e_{i,i+1}(h)+h\Phi\cdot I_S.
\]
This implies that
\[
\varphi\big(e_{ij}(h)\big)=e_{ij}(h)+h\Phi\cdot I_S,\qquad 1\le i<j\le n.
\]

There exist $\Phi$-linear mappings $\chi_{ij}\colon hR\to hS$ such that
\[
\varphi\bigl(e_{ij}(a)\bigr)=e_{ij}(\chi_{ij}(a))+h\Phi\cdot I_S,\qquad a\in hR,\quad 1\le i<j\le n.
\]
If $k<i$, then $e_{kj}(a)=[e_{ki}(h),e_{ij}(a)],$ which implies that $\chi_{kj}=\chi_{ij}$. If $q>j$, then, similarly, $\chi_{ij}=\chi_{iq}.$  Hence, all the mappings $\chi_{ij}$, where $1\le i<j\le n$, are equal to $\chi_{1n}=\chi$.

For an arbitrary element $a\in hR$ and arbitrary indices $1\le i<j\le n$, we have
\[
\varphi\big(e_{ij}(a)\big)=e_{ij}(\chi(a))+h\Phi\cdot I_S, \quad 
\chi(h)=h.
\]
Now, $[e_{12}(a),e_{23}(b)]=e_{13}(ab)$ implies that
\[
[e_{12}(\chi(a)),e_{23}(\chi(b))]=e_{13}(\chi(ab)),
\]
and, therefore, $\chi(a)\chi(b)=\chi(ab),$ for arbitrary elements $a,b\in hR$. Hence, $\chi$ is a homomorphism. This completes the proof of Lemma~\ref{Lm2}. \end{proof}

\begin{lemma}\label{Lm3}
Suppose that $f\in\Phi$ is an idempotent and
\[
\varphi(e_{11}(f))=
\left(
\begin{array}{cccc}
	0 & * & \cdots & * \\
	0 & 0 & \ddots & \vdots \\
	\vdots & \ddots & \ddots & * \\
	0 & \cdots & 0 & -f
\end{array}
\right)+f\Phi\cdot I_S.
\]
Then there exist
an antiisomorphism $\psi\colon fR\to fS$ and invertible
elements $g_1\in UT(n,S)$ and $ g_2\in \operatorname{Diag}(S^{*}),$ such
that $\varphi_h$ lifts to $$-\widehat{g}_1\,\widehat{g}_2\,\widetilde{\psi}.$$ \end{lemma}

\begin{proof}
As in the proof of Lemma~\ref{Lm1} (see 5)),
consider the composition
\[
T(R)/(\Phi\cdot I_R) \xrightarrow{\ \varphi\ } T(S)/(\Phi\cdot I_S)
\xrightarrow{\ \tau_S\ } T(S^{\mathrm{op}})/(\Phi\cdot I_{S^{\mathrm{op}}}),
\quad
\pi=\tau_S\varphi.
\]
The Lie isomorphism $-\pi$ maps $e_{11}(f)$
to
\[
\left(
\begin{array}{cccc}
	f & 0 & \cdots & 0 \\
	0 & 0 & \cdots & 0 \\
	\vdots & \ddots & \ddots & \vdots \\
	0 & \cdots & \cdots & 0
\end{array}
\right)+f\Phi\cdot I_{S^{\mathrm{op}}}.
\]
Now the assertion follows from Lemma~\ref{Lm2}. \end{proof}

Theorem\ref{Th_1} follows immediately from
Lemmas~\ref{Lm1}, \ref{Lm2}, and \ref{Lm3}.

\section{Jordan Isomorphisms}

\begin{lemma}\label{Lm4}
	Let $\varphi\colon T(R)\to T(S)$ be a Jordan isomorphism. Then there exists an idempotent
	$h\in \Phi$ such that
	\[
	\varphi\bigl(e_{11}(1_R)\bigr)=
\left(
\begin{array}{ccccc}
	h & * & \cdots & & * \\
	0 & 0 & \ddots & & \vdots \\
	\vdots & \ddots & \ddots & \ddots & \vdots \\
		0 & \cdots & \ddots & 0 & * \\
	0 & \cdots & \cdots & 0 & f
\end{array}
\right),
	\qquad f=1_S-h.
	\]
\end{lemma}

\begin{proof}	We have $\mathfrak{A}(R)=\{(a\circ b)\circ c-(a\circ c)\circ b \mid a,b,c\in T(R)\}.$	Hence,	$\varphi\bigl(\mathfrak{A}(R)\bigr)=\mathfrak{A}(S).$	Furthermore,
	\[
	e_{1n}(R)=\{a\in \mathfrak{A}(R)\mid a\circ \mathfrak{A}(R)=(0)\}. 
	\]
	Hence, $\varphi\bigl(e_{1n}(R)\bigr)=e_{1n}(S).$ 
	
	Let
	\[
	\varphi\bigl(e_{11}(1_R)\bigr)=
\left(
\begin{array}{cccc}
	\alpha_1 & * & \cdots & * \\
	0 &\ddots & \ddots & \vdots \\
	\vdots & \ddots & \ddots & * \\
	0 & \cdots & 0 & \alpha_n
\end{array}
\right),
	\qquad \alpha_i\in \Phi.
	\]
For an arbitrary element $e_{1n}(a)$, where $a\in R$, we have
\[
e_{11}(1_R)\circ e_{1n}(a)=e_{1n}(a).
\]
Hence, for an arbitrary element $b\in S$, we have
\[
\varphi\bigl(e_{11}(1_R)\bigr)\circ e_{1n}(b)
= e_{1n}\bigl((\alpha_1+\alpha_n)b\bigr)
= e_{1n}(b).
\]
This implies that $\alpha_1+\alpha_n=1_S$.

We have $e_{11}(1_R)\,\mathfrak{A}(R)\,e_{11}(1_R)=(0)$. Hence,
\[
\varphi\bigl(e_{11}(1_R)\bigr)\,\mathfrak{A}(S)\,\varphi\bigl(e_{11}(1_R)\bigr)=(0).
\]
For an arbitrary element $b\in S$ and indices $1\le i<j\le n$, we have
\[
\left(\begin{array}{ccc}
	\alpha_1  & \cdots & * \\
	\vdots & \ddots & \vdots \\
	0 & \cdots & \alpha_n
\end{array}\right)
e_{ij}(b)
\left(\begin{array}{ccc}
	\alpha_1  & \cdots & * \\
	\vdots & \ddots & \vdots \\
	0 & \cdots & \alpha_n
\end{array}\right)
\in
e_{ij}(\alpha_i\alpha_j b)+T_j^{+}(S).
\]
This implies that $\alpha_i\alpha_j=0$ for all $1\le i<j\le n$.
In particular, for $1<i<n$, we have $\alpha_i\alpha_1=\alpha_i\alpha_n=0.$ Therefore, $\alpha_i(\alpha_1+\alpha_n)=\alpha_i=0.$ This completes the proof of the lemma.
\end{proof}

\begin{lemma}\label{Lm5} Let \(h\in \Phi\) be an idempotent. Consider the restriction of the Jordan automorphism \(\varphi\):
\[
\varphi_h \colon hT(R)\to hT(S)
\]
and suppose that
\[
\varphi_h\bigl(e_{11}(h)\bigr)
=
\left(
\begin{array}{cccc}
	h & * & \cdots & * \\
	0 & 0 & \ddots & \vdots \\
	\vdots & \ddots & \ddots & * \\
	0 & \cdots & 0 & 0
\end{array}
\right).
\]
Then there exist an isomorphism \(\psi \colon hR \to hS\) and invertible elements \(g_1\in UT(n,S)\) and \(g_2\in \operatorname{Diag}(S^*)\) such that
\[
\varphi_h=\widehat{g}_1\,\widehat{g}_2\,\widetilde{\psi}.
\] \end{lemma}

\begin{proof} As in the proof of Lemma~\ref{Lm2}, we find an element \(g\in UT(n,S)\) such that 
\[
g\,\varphi\bigl(e_{11}(h)\bigr)\,g^{-1}
=
\begin{pmatrix}
	h & 0 & \cdots & 0 \\
	0   & 0 &   \cdots     & * \\
	\vdots &  & \ddots & \vdots \\
	0   & 0 & \cdots & 0
\end{pmatrix}.
\] Hence, without loss of generality, we assume
that
\[
\varphi\bigl(e_{11}(h)\bigr)=
\begin{pmatrix}
	h & 0 & \cdots & 0 \\
	0 & 0 & \cdots & * \\
	\vdots & \vdots & \ddots & \vdots \\
	0 &0 & \cdots & 0
\end{pmatrix}.
\]
 
Let \(1\le t\le n-1\). Suppose that condition \(P(t)\) holds:
\[
\varphi\bigl(e_{ii}(h)\bigr)=
\begin{tikzcd}[column sep=-0.5em, row sep=0pt, ampersand replacement=\&]
	\left(
	\begin{array}{c|c}
		e_{ii}(h) & 0 \\
		\hline
		0 &
		\begin{array}{ccc}
			0  & \cdots & * \\
			\vdots & \ddots & \vdots \\
			0 & \cdots & 0
		\end{array}
	\end{array}
	\right)
	\hspace{-0.6em}
	\begin{array}{c}
		\Bigl.\vphantom{
			\begin{array}{c}
				h\\ h\\ h
			\end{array}
		}\bigr\} {\scriptstyle t} 
		\\[+12.3ex]
	\end{array}
	\& {}
\end{tikzcd}
\qquad 1\le i\le t.
\]
By \cite{Jacobson2}, Jordan homomorphisms map pairwise orthogonal idempotents to pairwise orthogonal idempotents. Hence, for $1\le i\le t$, we obtain 
\[
\varphi\bigl(e_{t+1,t+1}(h)\bigr)
\begin{tikzcd}[column sep=-0.5em, row sep=0pt, ampersand replacement=\&]
	\left(
	\begin{array}{c|c}
		e_{ii}(h) & 0 \\
		\hline
		0 & *
	\end{array}
	\right)
	\hspace{-0.6em}
	\begin{array}{c}
		\Bigl.\vphantom{
			\begin{array}{c}
				h\\ h\\ h
			\end{array}
		}\bigr\} {\scriptstyle t} 
		\\[+5.1ex]
	\end{array}
	\&
\end{tikzcd}
=	\left(
\begin{array}{c|c}
	e_{ii}(h) & 0 \\
	\hline
	0 & *
\end{array}
\right)
\varphi\bigl(e_{t+1,t+1}(h)\bigr)
=0.\] 
This implies that
\[
\varphi\bigl(e_{t+1,t+1}(h)\bigr)=
\begin{tikzcd}[column sep=-0.5em, row sep=0pt, ampersand replacement=\&]
	\left(
	\begin{array}{c|c}
	0 & 0 \\
		\hline
		0 & *
	\end{array}
	\right)
	\hspace{-0.6em}
	\begin{array}{c}
		\Bigl.\vphantom{
			\begin{array}{c}
				h\\ h\\ h
			\end{array}
		}\bigr\} {\scriptstyle t} 
		\\[+5.1ex]
	\end{array}
	\&
\end{tikzcd}.
\]

Let
\[
\varphi\bigl(e_{t+1,t+1}(h)\bigr)=
\begin{tikzcd}[column sep=-0.5em, row sep=0pt, ampersand replacement=\&]
	\left(
	\begin{array}{c|c}
	0 & 0 \\
		\hline
		0 &
		\begin{array}{ccc}
			\gamma_1  & \cdots & * \\
			\vdots & \ddots & \vdots \\
			0 & \cdots & \gamma_{n-t}
		\end{array}
	\end{array}
	\right),
\end{tikzcd} \quad 
\gamma_i\in h\Phi,
\ \ 
1\le i\le n-t.
\]

For an arbitrary \(1\le i\le n-1\), we have
\[
\mathfrak{A}(R)^i=
\underbrace{\Big(\cdots\big((\mathfrak{A}(R)\circ \mathfrak{A}(R))\circ\mathfrak{A}(R)\big)\cdots\Big)\circ \mathfrak{A}(R)}_{i}=
\]
\[
\begin{array}{c}
	\hspace{-6.1em}\overbrace{\hspace{2.9em}}^{\,i-1} \\[-0.4em]
	\left(
	\begin{array}{ccccc}
		0 & \cdots & R & \cdots & R \\
	0	& \ddots & \ddots & \ddots & \vdots \\
	\vdots	& & \ddots & \ddots & R \\
		\vdots	& & & \ddots & 0 \\
	0	& \cdots & \cdots & \cdots & 0
	\end{array}
	\right)
\end{array}
\]
Hence,
\[
\bigl(\mathfrak{A}(R)\, \circ \, e_{11}(h)\bigl)\, \circ \, e_{t+1,t+1}(h)
=
e_{1,t+1}(hR)\subseteq \mathfrak{A}(R)^t,
\]
\[
e_{1,t+1}(hR)\cap \mathfrak{A}(R)^{t+1}=(0).
\]

For an arbitrary $t+1<j\le n$, we have
\[
\left(
e_{1j}(hS)\circ
		\left(\begin{array}{c|c}
	\begin{matrix}
	h & \cdots & \cdots & 0 \\
	0 & 0 &  \cdots & 0\\
\vdots & 	\vdots & \ddots & \vdots \\
	0 & 0 & \cdots & 0
\end{matrix} & 0 \\
	\hline
	0 &
	\begin{matrix}
		0 & \cdots & * \\
	\vdots & \ddots & \vdots \\
		0 & \cdots & 0
	\end{matrix}
\end{array} 	\right)
\right)
\circ
	\left(\begin{array}{c|c}
	0 & 0 \\
		\hline
	0 &
	\begin{matrix}
		\gamma_1 & \cdots & * \\
		\vdots & \ddots & \vdots \\
		0 & \cdots & \gamma_{n-t}
	\end{matrix}
\end{array}\right)
\in \] 
\[e_{1j}(\gamma_jS)+T_{1j}^{+}(S).
\]
Hence, $\gamma_{t+1}=\cdots=\gamma_n=0.$

The equality
\[
\bigl(\mathfrak{A}(R)^t\circ e_{t+1,t+1}(h)\bigr)\circ
\mathfrak{A}(R)^{\,n-t-1}
=
e_{1n}(hR)
\]
implies
\[
\mathfrak{A}(S)^t\circ
	\left(\begin{array}{c|c}
	0 & 0 \\
		\hline
	0 &
	\begin{matrix}
		\gamma_{t+1} & 0 & \cdots & 0 \\
		0 & 0 & \cdots & 0 \\
		\vdots & \ddots & \ddots & \vdots\\
	0 & \cdots & \cdots & 0
	\end{matrix}
	\end{array}
\right)
\circ
\mathfrak{A}(S)^{\,n-t-1}
=
e_n(hR).
\]
Therefore, the element \(\gamma_{t+1}\) is invertible in \(hS\). Since \(\gamma_{t+1}\) is an idempotent, we obtain $\gamma_{t+1}=h.$ Now,
\[
\varphi\bigl(e_{t+1,t+1}(h)\bigr)=
	\left(\begin{array}{c|c}
	0 & 0 \\
	\hline
	0 &
	\begin{matrix}
		h & * & \cdots & * \\
		0 & 0 & \cdots & * \\
		\vdots & \ddots & \ddots & \vdots\\
		0 & \cdots & \cdots & 0
	\end{matrix}
\end{array}
\right).
\]

As in the proof of Lemma~\ref{Lm2}, we find an element
\[
x\, \in
\begin{tikzcd}[column sep=-0.5em, row sep=0pt, ampersand replacement=\&]
	\left(
	\begin{array}{c|c}
	\begin{array}{ccc}
	1 & \cdots & 0 \\
		\vdots & \ddots & \vdots \\
		0 & \cdots & 1
	\end{array} & 0 \\
		\hline
		0 &
		\begin{array}{cccc}
			1 & S & \cdots & S \\
			0 & 1 & \cdots & 0 \\
				\vdots & \ddots & \ddots & \vdots \\
			0 & \cdots & \cdots & 1
		\end{array}
	\end{array}
	\right)
	\hspace{-0.6em}
	\begin{array}{c}
		\Biggl.\vphantom{
			\begin{array}{c}
				h\\ h\\ h
			\end{array}
		}\Biggr\} {\scriptstyle t} 
		\\[15.3ex]
	\end{array}
	\& {}
\end{tikzcd}
\]
such that
\[
x\,\varphi\bigl(e_{t+1,t+1}(h)\bigr)\,x^{-1}
=
\begin{tikzcd}[column sep=-0.5em, row sep=0pt, ampersand replacement=\&]
	\left(
	\begin{array}{c|c}
		e_{t+1, t+1}(h) & 0 \\
		\hline
		0 &
		\begin{array}{ccc}
			0 &  \cdots & * \\
		\vdots & \ddots &  \vdots \\
			0 & \cdots  & 0
		\end{array}
	\end{array}
	\right)
	\hspace{-0.6em}
	\begin{array}{c}
		\Bigl.\vphantom{
			\begin{array}{c}
				h\\ h\\ h
			\end{array}
		}\bigr\} {\scriptstyle t+1} 
		\\[+12.3ex]
	\end{array}
	\& . {}
\end{tikzcd}
\]

The idempotents
\[
x\,\varphi\bigl(e_{ii}(h)\bigr)\,x^{-1}
=
\begin{tikzcd}[column sep=-0.5em, row sep=0pt, ampersand replacement=\&]
	\left(
	\begin{array}{c|c}
		e_{ii}(h) & 0 \\
		\hline
		0 & *
	\end{array}
	\right)
	\hspace{-0.6em}
	\begin{array}{c}
		\Bigl.\vphantom{
			\begin{array}{c}
				h\\ h\\ h
			\end{array}
		}\bigr\} {\scriptstyle t} 
		\\[+5.5ex]
	\end{array},
	\& {}
\end{tikzcd}
\ \  1\le i\le t,
\]
are orthogonal to
\[
\begin{tikzcd}[column sep=-0.5em, row sep=0pt, ampersand replacement=\&]
	\left(
	\begin{array}{c|c}
		e_{t+1,t+1}(h) & 0 \\
		\hline
		0 & *
	\end{array}
	\right)
	\hspace{-0.6em}
	\begin{array}{c}
		\Bigl.\vphantom{
			\begin{array}{c}
				h\\ h\\ h
			\end{array}
		}\bigr\} {\scriptstyle t+1} 
		\\[+5.5ex]
	\end{array}.
	\& {}
\end{tikzcd}
\]
This implies that
\[
x\,\varphi\bigl(e_{ii}(h)\bigr)\,x^{-1}
=
\begin{tikzcd}[column sep=-0.5em, row sep=0pt, ampersand replacement=\&]
	\left(
	\begin{array}{c|c}
		e_{ii}(h) & 0 \\
		\hline
		0 & *
	\end{array}
	\right)
	\hspace{-0.6em}
	\begin{array}{c}
		\Bigl.\vphantom{
			\begin{array}{c}
				h\\ h\\ h
			\end{array}
		}\bigr\} {\scriptstyle t+1} 
		\\[+5.5ex]
	\end{array}.
	\& {}
\end{tikzcd}
\]

We have proved that if the Jordan isomorphism \(\varphi\) satisfies \(P(t)\), then there exists an element \(g\in UT(n,S)\) such that \(\widehat{g}\,\varphi\) satisfies \(P(t+1)\). Hence, without loss of generality, we may assume that \(\varphi\) satisfies \(P(n)\), that is, $\varphi\bigl(e_{ii}(h)\bigr)=e_{ii}(h)$ for   $1\le i\le n.$ As in the proof of Lemma~\ref{Lm2}, this implies that
\[
\varphi\bigl(e_{ij}(hR)\bigr)
=
\varphi\Bigl((T(R)\circ e_{ii}(h))\circ e_{jj}(h)\Bigr)
=
e_{ij}(hS), \quad 1\le i<j\le n.
\]

Let \(\varphi\bigl(e_{ij}(h)\bigr)=e_{ij}(u_{ij})\), where \(u_{ij}\in hS\) and \(1\le i<j\le n\). The equality
\[
e_{1n}(hR)=\bigl(e_{1i}(h)\circ e_{ij}(h)\bigr)\circ e_{jn}(hR)
\]
implies
\[
e_{1n}(hS)
=
\bigl(e_{1i}(u_{1i})\circ e_{ij}(u_{ij})\bigr)\circ e_{jn}(hS)
=
e_{1n}(u_{1i}u_{ij}S).
\]
Hence, all elements \(u_{ij}\) are invertible in \(hS\).

As in Sec.~\ref{Sec_Lie}, we consider the diagonal matrix \(x\in \operatorname{Diag}(S^*)\),
\[
x=\operatorname{diag}(1_S,\;u_{12},\;u_{12}u_{23},\;\ldots,\;u_{12}u_{23}\cdots u_{n-1,n}).
\]
Then
\[
x\,\varphi\bigl(e_{i,i+1}(h)\bigr)x^{-1}=e_{i,i+1}(h),
\qquad 1\le i\le n-1,
\]
and therefore
\[
x\,\varphi\bigl(e_{ij}(h)\bigr)x^{-1}=e_{ij}(h),
\qquad 1\le i<j\le n.
\]

As in Sec.~\ref{Sec_Lie}, we consider the diagonal matrix \(x\in \operatorname{Diag}(S^*)\),
\[
x=\operatorname{diag}(1_S,\;u_{12}+f,\;u_{12}u_{23}+f,\;\ldots,\;u_{12}u_{23}\cdots u_{n-1,n}+f).
\]
Then
\[
x\,\varphi\bigl(e_{i,i+1}(h)\bigr)x^{-1}=e_{i,i+1}(h),
\qquad 1\le i\le n-1,
\]
and therefore
\[
x\,\varphi\bigl(e_{ij}(h)\bigr)x^{-1}=e_{ij}(h),
\qquad 1\le i<j\le n.
\]
Replacing \(\varphi\) with \(\widehat{x}\,\varphi\), we may assume that
\[
\varphi\bigl(e_{ij}(h)\bigr)=e_{ij}(h),
\qquad 1\le i<j\le n.
\]
By literally following the arguments of Sec.~\ref{Sec_Lie}, we obtain an isomorphism \(\chi\colon hR\to hS\) such that
\[
\varphi\bigl(e_{ij}(a)\bigr)=e_{ij}\bigl(\chi(a)\bigr),
\quad a\in hR,\quad  1\le i<j\le n.
\]
This completes the proof of Lemma~\ref{Lm5}. \end{proof}

\begin{lemma}\label{Lm6}
Suppose that \(f\in \Phi\) is an idempotent
and that
\[
\varphi\bigl(e_{11}(f)\bigr)=
\begin{pmatrix}
	0 & * & \cdots & * \\
	\vdots & \ddots & \ddots & \vdots \\
		0 & \cdots & 0 & * \\
	0 & \cdots & 0 & f
\end{pmatrix}.
\]
Then there exist an antiisomorphism $\psi\colon fR\to fS$ and invertible elements $g_1\in UT(n,S)$ and  $g_2\in \operatorname{Diag}(S^*)$  such that
\[
\varphi_h=\widehat{g}_1\,\widehat{g}_2\,\widetilde{\psi}.
\] \end{lemma}

\begin{proof}
As in the proof of Lemma~\ref{Lm3}, it is sufficient to apply Lemma~\ref{Lm5} to the composition
\[
T(R)\xrightarrow{\ \varphi\ }T(S)\xrightarrow{\ \tau_S\ }T(S^{\mathrm{op}}).
\] \end{proof}

Theorem~\ref{Th_2} follows immediately from Lemmas~\ref{Lm4}, \ref{Lm5}, and \ref{Lm6}.

		\section*{Acknowledgment}

\medskip

The author is grateful to Matej Bre\v{s}ar, Iryna Kashuba, and Efim Zelmanov for helpful discussions.

\end{document}